\documentclass[10pt]{amsart}
\usepackage{fancyhdr}
\usepackage{stmaryrd}
\usepackage{calrsfs}

\usepackage{amsmath}
\usepackage[nospace,noadjust]{cite}
\usepackage{amsfonts}
\usepackage{amssymb,enumerate}
\usepackage{amsthm}
\usepackage{cite}
\usepackage{comment}
\usepackage{color}
\usepackage[all]{xy}
\usepackage{hyperref}
\usepackage{lineno}

\newtheorem*{maintheorem*}{Main Theorem}
\newtheorem{theorem}{Theorem}[section]

\newtheorem{question}[theorem]{Question}

\newtheorem{lemma}[theorem]{Lemma}
\newtheorem{cor}[theorem]{Corollary}

\theoremstyle{definition}
\newtheorem{definition}[theorem]{Definition}
\newtheorem{remark}[theorem]{Remark}
\newtheorem{example}[theorem]{Example}
\numberwithin{equation}{section}

\newcommand{\ff}{\mathbb{F}}
\newcommand{\nn}{\mathbb{N}}
\newcommand{\pp}{\mathbb{P}}
\newcommand{\qq}{\mathbb{Q}}
\newcommand{\rr}{\mathbb{R}}
\newcommand{\zz}{\mathbb{Z}}

\newcommand{\gp}{\text{gp}}

\newcommand{\uu}{\mathcal{U}}

\providecommand\ldb{\llbracket}
\providecommand\rdb{\rrbracket}

\hypersetup{
	pdftitle={FD},
	colorlinks=true,
	linkcolor=blue,
	citecolor=cyan,
	urlcolor=wine
}



\begin{document}
	
	\mbox{}
	\title{Arithmetic properties encoded in undermonoids}
	
	\author{Felix Gotti}
	\address{Department of Mathematics\\MIT\\Cambridge, MA 02139}
	\email{fgotti@mit.edu}
	
	\author{Bangzheng Li}
	\address{Department of Mathematics\\MIT\\Cambridge, MA 02139}
	\email{liben@mit.edu}
	
\date{\today}

\begin{abstract}
	Let $M$ be a cancellative and commutative monoid. A submonoid $N$ of $M$ is called an undermonoid if the Grothendieck groups of $M$ and $N$ coincide. For a given property $\mathfrak{p}$, we are interested in providing an answer to the following main question: does it suffice to check that all undermonoids of $M$ satisfy $\mathfrak{p}$ to conclude that all submonoids of $M$ satisfy $\mathfrak{p}$? In this paper, we give a positive answer to this question for the property of being atomic, and then we prove that if $M$ is hereditarily atomic (i.e., every submonoid of $M$ is atomic), then $M$ must satisfy the ACCP, proving a recent conjecture posed by Vulakh and the first author. We also give positive answers to our main question for the following well-studied factorization properties: the bounded factorization property, half-factoriality, and length-factoriality. Finally, we determine all the monoids whose submonoids/undermonoids are half-factorial (or length-factorial).
\end{abstract}
\medskip

\maketitle


\bigskip
\section{Introduction}
\label{sec:intro}
\smallskip

Following Cohn~\cite{pC68}, we say that a cancellative and commutative monoid is atomic if every non-invertible element can be factored into atoms, often referred to as irreducible elements. On other hand, an integral domain is called atomic if its multiplicative monoid is atomic. Atomic monoids and domains are the fundamental playground for the study of arithmetic of factorizations, and so they have been systematically investigated under the umbrella of factorization theory. Such a systematic study began with the appearance of the highly-cited papers \cite{AAZ90, AAZ92} by Anderson, Anderson, and Zafrullah, where they consider factorizations in the setting of integral domains, and the follow up paper \cite{fHK92} by Halter-Koch, where he generalizes and studies the bounded and finite factorization properties (both introduced in~\cite{AAZ90}) to the setting of monoids. Atomicity has been actively studied recently (see~\cite{CG19,GR25} and references therein). For a survey on atomicity, see~\cite{CG24}.
\smallskip

The notion of hereditary atomicity was first investigated by Coykendall, Hasenauer, and the first author~\cite{CGH21} in the classical setting of integral domains, and then it was studied by Vulakh and the first author~\cite{GV23} in the more general setting of commutative monoids. A commutative monoid (resp., an integral domain) is called hereditarily atomic if all its submonoids (resp., subrings) are atomic. It was conjectured in \cite[Conjecture~6.1]{CGH21} that every hereditarily atomic integral domain satisfies the ascending chain condition on principal ideals (ACCP). The similar statement of this conjecture in the context of torsion-free monoids was settled in~\cite[Theorem~3.1]{GV23}, and in the same paper the authors conjectured that similar result should hold in the more general class of commutative monoids (without the torsion-free condition). In this paper, we give a positive answer to this conjecture: every hereditarily atomic monoid satisfies the ACCP (Theorem~\ref{thm:HA implies ACCP}). It is worth emphasizing that the connection between atomicity and the ACCP has received a fair amount of attention recently: see, for instance, \cite{BC19,fG23,GL23}.
\smallskip

In the study of hereditary atomicity, it is natural to consider the related notion of underatomicity, which we introduce here. A submonoid $N$ of a cancellative and commutative monoid $M$ is called an undermonoid if the Grothendieck\ groups of $M$ and $N$ coincide. Then we say that $M$ is underatomic if all its undermonoids are atomic. It is clear that every hereditarily atomic monoid is underatomic. Although the converse is not clear, it is also true. As the first main results of this paper, we prove that the notions of hereditary atomicity and underatomicity are equivalent (Theorem~\ref{thm:HA=UA}). As a result, we obtain that, in the class of reduced monoids, both hereditary atomicity and underatomicity are conditions equivalent to that of satisfying the ACCP. 
\smallskip

An atomic monoid is called a bounded factorization monoid (or a BFM for short) if for any element there is a common upper bound for the lengths (i.e., number of irreducibles counting repetitions) of any factorization of the element. As mentioned before, the bounded factorization property was introduced in~\cite{AAZ90} and has been systematically investigated by many authors ever since (for a recent survey on the bounded and finite factorization properties in the setting of integral domains, see~\cite{AG22}). As for the case of atomicity, we also establish here that in order to conclude that every submonoid of a given monoid is a BFM it suffices to argue that every undermonoid of the given monoid is a BFM. An atomic monoid is called half-factorial (resp., length-factorial) if any two factorizations of the same element must have equal (resp., distinct) length. The term `half-factoriality' was coined by Zaks-\cite{aZ76} back in the seventies in the context of Dedekind domains but it had been previously studied by Carlitz~\cite{lC60} in 1960 in the context of algebraic number fields. The notion of length-factoriality was introduced and studied more recently by Coykendall and Smith in~\cite{CS11} under the term `other-half-factoriality' (the term `length-factoriality was later adopted in~\cite{CCGS21}'). For both the half-factorial and length factorial monoids, we also prove that it suffices to check that the same properties hold for any undermonoid to conclude that the hold for actually all submonoids.
\smallskip

This paper is structured as follows. In order to make our exposition as self-contained as possible, in Section~\ref{sec:background} we briefly revise the terminology, notation, and results that we will use in the rest of this paper. In Section~\ref{sec:HA=UA}, we prove that that if a monoid is underatomic, then it is hereditarily atomic (Theorem~\ref{thm:HA=UA}). In Section~\ref{sec:connection to the ACCP}, we turn our attention to the ACCP, and give a positive answer to~\cite[Conjecture~3.5]{GV23}: we prove that every hereditarily atomic monoid satisfies the ACCP (Theorem~\ref{thm:HA implies ACCP}). In the special class of reduced cancellative and commutative monoids, this implies that being hereditarily atomic and satisfying the ACCP are equivalent conditions. In Section~\ref{sec:BF and FF}, we consider the bounded factorization property, and we argue that if every undermonoid of a given monoid is a BFM, then every submonoid of the same monoid is a BFM (Theorem~\ref{thm:under-BFM = hereditarily BFM}). We were unable to establish a similar result for the finite factorization property, and we highlight this as an open question for interested readers (Question~\ref{quest:under and hereditary FFM}). In Section~\ref{sec:HFM and LFM}, we prove that for any given monoid, if all its undermonoids are half-factorial (resp., length-factorial) monoids, then all its submonoids are half-factorial (resp., length-factorial) monoids (Theorems~\ref{thm:undermonoids and the HFM} and~\ref{thm:undermonoids and the LFM}). At the end of each section, we emphasize that the studied hereditary properties are not equivalent to the corresponding ring theoretical hereditary properties when we restrict to the class of integral domains.

\bigskip
\section{Background}
\label{sec:background}

In this section, we introduce most of the relevant concepts related to commutative monoids and factorization theory that are required to follow the results we will presented later. Background in commutative monoids and semigroups can be found in~\cite{pG01}, while background in atomicity and factorization theory can be found in~\cite{GH06}.

\smallskip
\subsection{General Notation}

We let $\nn$ denote the set of positive integers, and we set $\nn_0 := \{0\} \cup \nn$. In addition, we denote the set of primes by $\pp$. As it is customary, we let $\zz, \qq$, and $\rr$ denote the sets of integers, rational numbers, and real numbers, respectively. If $b,c \in \zz$ with $b \le c$, then we let $\ldb b,c \rdb$ denote the discrete interval from $b$ to $c$; that is,
\[
	\ldb b,c \rdb := \{n \in \zz : b \le n \le c\}.
\]
For a subset $R$ of $\rr$ and $\ell \in \rr$, we set $R_{\ge \ell} := \{r \in R : r \ge \ell\}$ and, in a similar manner, we use the notation $R_{> \ell}$. For $q \in \qq_{> 0}$, we let $\mathsf{n}(q)$ and $\mathsf{d}(q)$ denote the unique positive integers such that $q = \mathsf{n}(q)/\mathsf{d}(q)$ and $\gcd(\mathsf{n}(q), \mathsf{d}(q)) = 1$. If $Q$ is a subset of $\qq_{> 0}$, then we set $\mathsf{d}(Q) := \{\mathsf{d}(q) : q \in Q\}$. 
For disjoint sets $S$ and~$T$ we often write $S \sqcup T$ instead of $S \cup T$ to emphasize that we are taking the union of disjoint sets.

\smallskip
\subsection{Commutative Monoids}

Throughout this paper, we reserve the term \emph{monoid} for a cancellative and commutative semigroup with identity and, unless otherwise is clear from the context, we will use additive notation for monoids. Let $M$ be a monoid. We set $M^\bullet := M \setminus \{0\}$ and, as for groups, we say that $M$ is \emph{trivial} if $M = \{0\}$. We let $\uu(M)$ denote the group of invertible elements of $M$. If $\uu(M)$ is trivial, then we say that $M$ is \emph{reduced}. The \emph{Grothendieck group} of $M$ is denoted by $\gp(M)$; that is, $\gp(M)$ is the unique abelian group up to isomorphism satisfying that any abelian group containing an isomorphic image of $M$ also contains an isomorphic image of $\gp(M)$. 
The monoid $M$ is \emph{torsion-free} if $\gp(M)$ is a torsion-free group. The \emph{order} of an element of $M$ is its order as an element of $\gp(M)$. Observe that if every element of $M$ has finite order, then $M$ must be a group. A submonoid $N$ of $M$ is called an \emph{undermonoid} if $\gp(N) = \gp(M)$, where $\gp(N)$ denotes the natural copy of the Grothendieck group of $N$ contained in $\gp(M)$. Undermonoids play a central role in the context of this paper. 

\smallskip

For $b,c \in M$, we say that $b$ \emph{divides} $c$ \emph{in} $M$ and write $b \mid_M c$ provided that $c = b + b'$ for some $b' \in M$ (we write $b \nmid_M c$ when $b$ does not divide $c$ in $M$). 
For a subset $S$ of $M$, we let $\langle S \rangle$ denote the smallest submonoid of $M$ containing $S$. 
\smallskip

An element $a \in M \setminus \uu(M)$ is called an \emph{atom} if whenever $a = b+c$ for some $b,c \in M$ either $b \in \uu(M)$ or $c \in \uu(M)$. We let $\mathcal{A}(M)$ denote the set consisting of all the atoms of $M$. Note that if $M$ is reduced, then $\mathcal{A}(M)$ must be contained in every generating set of~$M$. 
An element of $b \in M$ is called \emph{atomic} if either $b$ is invertible or if $b$ can be written as a sum of atoms. Following \cite{pC68}, we say that $M$ is \emph{atomic} if every element of $M$ is atomic. In addition, following Coykendall et al.~\cite{CGH21}, we say that $M$ is \emph{hereditarily atomic} if every submonoid of $M$ is atomic. 
\smallskip

A subset $I$ of $M$ is called an \emph{ideal} of~$M$ if $I + M \subseteq I$ or, equivalently, if $I + M = I$. If~$I$ is an ideal of~$M$ such that $I = b + M := \{b\} + M$ for some $b \in M$, then $I$ is called \emph{principal}. The monoid $M$ satisfies the \emph{ascending chain condition on principal ideals} (ACCP) if every ascending chain $(b_n + M)_{n \ge 1}$ of principal ideals of $M$ eventually stabilizes; that is, there is an $n_0 \in \nn$ such that $b_n + M = b_{n+1} + M$ for every $n \ge n_0$. Every monoid satisfying the ACCP is atomic \cite[Proposition~1.1.4]{GH06}. The converse does not hold. The first study of the connection of atomicity and the ACCP was carried out in 1974 by Grams in~\cite{aG74}, where she constructed the first example of an atomic integral domain that does not satisfy the ACCP. For a generous insight of the connection of atomicity and the ACCP, see \cite{GL23a} and references therein. It was recently proved in~\cite[Theorem~3.1]{GV23} that every hereditarily atomic torsion-free monoid satisfies the ACCP. 
\smallskip

\subsection{Factorizations}

Let $M$ be an atomic monoid. The free commutative monoid on $\mathcal{A}(M/\uu(M))$ is denoted by $\mathsf{Z}(M)$, and the elements of $\mathsf{Z}(M)$ are called \emph{factorizations}. If $z := a_1 + \cdots + a_\ell$ is a factorization in $\mathsf{Z}(M)$ for some $\ell \in \nn$ and $a_1, \dots, a_\ell \in \mathcal{A}(M/\uu(M))$, then~$\ell$ is called the \emph{length} of $z$; we often write $|z|$ instead of $\ell$. Let $\phi \colon \mathsf{Z}(M) \to M/\uu(M)$ be the unique monoid homomorphism satisfying $\phi(a) = a$ for all $a \in \mathcal{A}(M/\uu(M))$. For each $b \in M$, we set
\[
	\mathsf{Z}(b) := \mathsf{Z}_M(b) := \phi^{-1}(b + \uu(M)) \subseteq \mathsf{Z}(M).
\]
Observe that $M$ is atomic if and only if $\mathsf{Z}(b)$ is nonempty for all $b \in M$ (notice that $\mathsf{Z}(u) = \{\emptyset\}$ for every $u \in \uu(M)$). The monoid $M$ is called a \emph{unique factorization monoid} (UFM) if $|\mathsf{Z}(b)| = 1$ for all $b \in M$, while $M$ is called a \emph{finite factorization monoid} (FFM) if $1 \le |\mathsf{Z}(b)| < \infty$ for all $b \in M$. For each $b \in M$, we set
\[
	\mathsf{L}(b) := \mathsf{L}_M(b) := \{|z| : z \in \mathsf{Z}(b)\}.
\]
The monoid $M$ is called a \emph{half-factorial monoid} (HFM) if $|\mathsf{L}(b)| = 1$ for all $b \in M$, while $M$ is called a \emph{bounded factorization monoid} (BFM) if $1 \le |\mathsf{L}(b)| < \infty$ for all $b \in M$. It follows directly from the definitions that every UFM is both an FFM and an HFM, and also that if a monoid is either an FFM or an HFM, then it has to be a BFM. In addition, one can readily show that every BFM satisfies the ACCP \cite[Corollary~1.3.3]{GH06}. See~\cite{AG22}, for a recent survey on the bounded and finite factorization properties focused on the setting of integral domains.
\smallskip

At various points throughout this paper we consider subrings and underrings of integral domains with the purpose of emphasizing that, unlike the case of many classical factorization properties, the hereditary properties we consider here do not immediately imply the corresponding properties for integral domains. Let $R$ be an integral domain. We let $R^*$ denote the multiplicative monoid of $R$, that is, the multiplicative monoid consisting of all nonzero elements of $R$. We say that $R$ is \emph{atomic} (resp., a BFD, an HFD, or an FFD) provided that the monoid $R^*$ is atomic (resp., a BFM, an HFM, or an FFM).

\bigskip
\section{Underatomicity and Hereditary Atomicity}
\label{sec:HA=UA}

Let $M$ be a monoid. The primary purpose of this section is to determine whether the fact that every undermonoid of $M$ is atomic suffices to conclude that $M$ is hereditarily atomic. With this in mind, we make the following definition.

\begin{definition}
	A monoid $M$ is \emph{underatomic} if every undermonoid of $M$ is atomic.
\end{definition}

It follows directly from the definitions that every hereditarily atomic monoid is underatomic. As the main result of this section, we prove that the notions of underatomicity and hereditary atomicity are indeed equivalent. In order to establish this main result, we need to introduce an auxiliary poset based on a given monoid and a fixed element. Fix an element $b \in M$ that is not atomic in some submonoid of~$M$, and consider the following set:
\begin{equation} \label{eq:aux poset}
	\mathcal{S}_b := \big\{ S : S \text{ is a submonoid of } M \text{ and } b \in S \text{ is not atomic in } S \big\}.
\end{equation}
The set $\mathcal{S}_b$ is, therefore, nonempty. Now we consider the binary relation $\preceq_b$ on the set~$\mathcal{S}_b$ defined as follows: for any submonoids $S_1$ and $S_2$ contained in $\mathcal{S}_b$,
\begin{equation} \label{eq:aux binary relation}
	S_1 \preceq_b S_2 \quad \text{if} \quad S_1 \subseteq S_2 \ \ \text{and} \ \ \uu(S_1) = S_1 \cap \uu(S_2).
\end{equation}

\begin{lemma} \label{lem:poset with Zorn's lemma}
	Let $M$ be a monoid, and let $b$ be an element of $M$ that is not atomic in some submonoid of~$M$. If $\mathcal{S}_b$ and $\preceq_b$ are defined as in \eqref{eq:aux poset} and~\eqref{eq:aux binary relation}, respectively, then the following statements hold.
	\begin{enumerate}
		\item $\mathcal{S}_b$ is a nonempty poset under the binary relation $\preceq_b$.
		\smallskip
		
		\item $\mathcal{S}_b$ has a maximal element. 
	\end{enumerate}
\end{lemma}

\begin{proof}
	To ease the notation of this proof, let $\mathcal{S}$ and $\preceq$ stand for $\mathcal{S}_b$ and $\preceq_b$, respectively.
	\smallskip
	
	(1) It is clear that $\preceq$ is reflexive and antisymmetric. Therefore proving that $\preceq$ is a partial order on $\mathcal{S}$ amounts to checking that $\preceq$ is transitive. To see this, suppose that $S_1 \preceq S_2$ and $S_2 \preceq S_3$ for some $S_1,S_2, S_3 \in \mathcal{S}$. It is clear that $S_1 \subseteq S_3$. In addition, observe that $\uu(S_1) \subseteq S_1 \cap \uu(S_3)$. For the reverse inclusion, pick $x \in S_1 \cap \uu(S_3)$. As $S_2 \preceq S_3$, it follows that $x \in S_2 \cap \uu(S_3) = \uu(S_2)$. This implies that $x \in S_1 \cap \uu(S_2) = \uu(S_1)$ because $S_1 \preceq S_2$, which in turn allows us to conclude that $S_1 \cap \uu(S_3) \subseteq \uu(S_1)$. As a result, $\preceq$ is transitive, and so $\preceq$ turns $\mathcal{S}$ into a nonempty poset.
	\smallskip
	
	(2) Let us argue that every nonempty chain in the poset $\mathcal{S}$ has an upper bound. For this, let $(S_i)_{i \in I}$ be a chain in~$\mathcal{S}$ for some nonempty set of indices $I$. We claim that
	\[
		S := \bigcup_{i \in I} S_i
	\]
	is an upper bound for the chain $(S_i)_{i \in I}$ in~$\mathcal{S}$. From the fact that $(S_i)_{i \in I}$ is a chain, we can immediately deduce that $S$ is a submonoid of $M$. Let us verify now that, for each $i \in I$, the equality
	\begin{equation} \label{eq:aux equality}
		\uu(S_i) = S_i  \cap \uu(S)
	\end{equation}
	holds. To do so, fix $i \in I$: as it is clear that $\uu(S_i) \subseteq S_i  \cap \uu(S)$, it suffices to show that the inclusion $S_i  \cap \uu(S) \subseteq \uu(S_i)$ holds. Pick $x \in S_i  \cap \uu(S)$. Then there exists $y \in S$ such that $x+y = 0$. After taking $j \in I$ such that $x,y \in S_j$ and $S_i \preceq S_j$, we see that $x \in S_i \cap \uu(S_j) = \uu(S_i)$. Hence the inclusion $S_i  \cap \uu(S) \subseteq \uu(S_i)$ holds, as desired. 
	
	Now suppose, by way of contradiction, that $S$ does not belong to $\mathcal{S}$. Since $S$ is a submonoid of $M$ and $b \in S$ (indeed, $b \in S_i$ for all $i \in I$), $b$ must be a non-atomic element of $S$. Since $b \in S_i \setminus \uu(S_i)$ for any $i \in I$, it follows from~\eqref{eq:aux equality} that $b \notin \uu(S)$. Therefore we can write $b = a_1 + \dots + a_n$ for some $a_1, \dots, a_n \in \mathcal{A}(S)$. As $(S_i)_{i \in I}$ is a chain, we can pick an index $k \in I$ such that $a_1, \dots, a_n \in S_k$. We claim that $a_1, \dots, a_n \in \mathcal{A}(S_k)$. To prove our claim, fix $\ell \in \ldb 1,n \rdb$, and let us argue that $a_\ell \in \mathcal{A}(S_k)$. The fact that $a_\ell \notin \uu(S)$ guarantees that $a_\ell \notin \uu(S_k)$. Now write $a_\ell = c+d$ for some $c,d \in S_k$. Because $a_\ell \in \mathcal{A}(S)$, either $c \in \uu(S)$ or $d \in \uu(S)$. If we assume that $c \in \uu(S)$, then in light of~\eqref{eq:aux equality}, we obtain that $c \in S_k \cap \uu(S) = \uu(S_k)$. We can similarly conclude that $d \in \uu(S_k)$ if we assume that $d \in \uu(S)$. Therefore $a_\ell \in \mathcal{A}(S_k)$, as desired. Hence $a_1, \dots, a_n \in \mathcal{A}(S_k)$, and so the equality $b = a_1 + \dots + a_n$ contradicts that $b$ is not an atomic element of $S_k$. As a consequence, $S$ belongs to $\mathcal{S}$.
	
	Since $S$ belongs to $\mathcal{S}$, it follows from~\eqref{eq:aux equality} that $S_i \preceq S$ for all $i \in I$, whence $S$ is an upper bound of the chain $(S_i)_{i \in I}$ in $\mathcal{S}$. Thus, $\mathcal{S}$ is a poset satisfying the condition in Zorn’s lemma, and so $\mathcal{S}$ contains a maximal element~$S$. 
\end{proof}

We are in a position to establish the primary result of this section.

\begin{theorem} \label{thm:HA=UA}
	For a monoid $M$, the following conditions are equivalent.
	\begin{enumerate}
		\item[(a)] $M$ is underatomic.
		\smallskip
		
		\item[(b)] $M$ is hereditarily atomic.
	\end{enumerate}
\end{theorem}

\begin{proof}
	(b) $\Rightarrow$ (a): This follows directly from the definitions.
	\smallskip
	
	(a) $\Rightarrow$ (b): Assume now that $M$ is underatomic. Suppose, by way of contradiction, that $M$ is not hereditarily atomic. We split the rest of the proof into the following two cases.
	\smallskip
	
	\noindent \textsc{Case 1:} $M$ is not a group. Let $N$ be a submonoid of $M$ that is not atomic, and take a non-invertible element $b \in N$ that does not factor into atoms in $N$. Now consider the set
	\[
		I := \{m \in M : m \nmid_M b \}.
	\]
	Notice that if $m \in M$ is not invertible, then $b + m \nmid_ M b$, and so $b + m \in I$. Thus, the fact that~$M$ is not a group guarantees that $I$ is nonempty. In addition, observe that $I$ cannot contain any invertible element of $M$. Lastly, it is clear that $I + M \subseteq I$. As a result, $I$ is a nonempty ideal of $M$ such that $I \subseteq M \setminus \uu(M)$.
	
	Now set $M' := N \cup I$. We claim that $M'$ is an undermonoid of $M$. Since $N$ is a submonoid and $I$ is an ideal, $M'$ is a submonoid of $M$. Now consider an arbitrary element of $\gp(M)$ and write it as $c-d$ for some $c,d \in M$. Since $I$ is nonempty, we can take $e \in I$ and observe that the fact that $e + c, e+d \in I$ ensures that $c-d = (e+c) - (e+d) \in I - I \subseteq \gp(M')$. Hence $\gp(M') = \gp(M)$, and so $M'$ is an undermonoid of~$M$, as claimed.
	
	Clearly, $\uu(N)\subseteq \uu(M')$. The reverse inclusion also holds because $I \cap \uu(M)$ is empty. Hence $\uu(M') = \uu(N)$. This in turn implies that $\mathcal{A}(M') \cap N \subseteq \mathcal{A}(N)$. We proceed to argue that $M'$ is not atomic, which amounts to showing that $b$ cannot be written as a sum of atoms in $M'$ (recall that $b \in N \subseteq M'$). Indeed, if $b = a'_1 + \dots + a'_\ell$ for some $a_1, \dots, a'_\ell \in \mathcal{A}(M')$, then the fact that $a'_i \mid_M b$ for every $i \in \ldb 1,\ell \rdb$ would imply that $a'_1, \dots, a'_\ell \notin I$, which in turn would imply that
	\[
		a'_1, \dots, a'_\ell \in \mathcal{A}(M') \cap N \subseteq \mathcal{A}(N).
	\]
	However, this is not possible because~$b$ cannot be written as a sum of atoms in $N$. As a consequence, the monoid $M'$ is not atomic. Finally, the fact that $M'$ is an undermonoid of $M$ contradicts the underatomicity of $M$.
	\smallskip
	
	\noindent \textsc{Case 2:} $M$ is a group. Since $M$ is not hereditarily atomic, there exists a submonoid $N$ of $M$ and a non-invertible element $b \in N$ that is not atomic in $N$. Let $\mathcal{S}_b$ and $\preceq_b$ be as in~\eqref{eq:aux poset} and \eqref{eq:aux binary relation}, respectively, and in order to simplify our notation set $\mathcal{S} := \mathcal{S}_b$ and $\preceq \, := \, \preceq_b$. It follows from Lemma~\ref{lem:poset with Zorn's lemma} that~$\mathcal{S}$ is a nonempty poset under the binary relation $\preceq$ and also that the poset $\mathcal{S}$ has a maximal element~$S$. Therefore $b \in S$ but $b$ is not atomic in $S$, which implies that $b \in S \setminus \uu(S)$. 

	Our next step is to argue that if there exists $u \in M$ satisfying at least one of the conditions~(i) and~(ii) below, then $2b + u \in S$:
	\begin{enumerate}
		\item[(i)] there exists $n_0  \in \nn$ with $n_0 u \in S$;
		\smallskip
		
		\item[(ii)] $nu+q \neq 0$ for any $n \in \nn$ and $q \in S$.
	\end{enumerate}
	To prove this, take $u \in M$ satisfying at least one of the conditions~(i) and~(ii), set $S' := S + \nn_0(2b + u)$, and let us prove the following claims.
	\smallskip

	\noindent \textsc{Claim 1.} $\uu(S') = \uu(S)$ and $2b+u \notin \uu(S')$.
	\smallskip
	
	\noindent \textsc{Proof of Claim 1.} First, we will prove that $2b+u \notin \uu(S')$. Suppose, towards a contradiction, that this is not the case, and take $s \in S$ and $k \in \nn_0$ such that the equality $2b+u + s +k(2b+u)=0$ holds. Since $(k+1)u + (s + (2k+2)b) = 0$, condition~(ii) above does not hold, and so there must exist $n_0 \in \nn$ such that $m_0 := n_0 u \in S$.  Then, after multiplying both sides of the equality~$2b+u + s +k(2b+u)=0$ by $n_0$, we find that $(k+1)m_0 + n_0 s + (2k + 2)n_0 b = 0$, which contradicts that $b \notin \uu(S)$. Thus, $2b+u \notin \uu(S')$.
	
	Given that $2b+u \notin \uu(S')$, arguing the equality $\uu(S') = \uu(S)$ amounts to proving the inclusion $S \cap \uu(S') \subseteq \uu(S)$. To do so, pick $s \in S \cap \uu(S')$. Then there exist $t \in S$ and $k \in \nn_0$ such that $s + t + k(2b+u) = 0$. If $k=0$, then $s \in \uu(S)$, as desired. Therefore suppose towards a contradiction that $k \ge 1$. Since $ku + (s + t + 2kb) = 0$, condition~(ii) above does not hold, and so there exists $n_0 \in \nn$ such that $m_0 := n_0 u \in S$. Then after multiplying both sides of the equality $s + t + k(2b+u) = 0$ by $n_0$, we find that $n_0s + n_0t + k m_0 + 2kn_0 b = 0$, contradicting that $b \notin \uu(S)$. Hence $S \cap \uu(S') \subseteq \uu(S)$, and so $\uu(S') = \uu(S)$. Claim~1 is now established.
	\smallskip
	
	\noindent \textsc{Claim 2.} $\mathcal{A}(S') \subseteq \mathcal{A}(S) \cup (2b + u + \uu(S))$.
	\smallskip
	
	\noindent \textsc{Proof of Claim 2.} We first show that $\mathcal{A}(S') \cap S \subseteq \mathcal{A}(S)$. To do this, take $a \in \mathcal{A}(S') \cap S$. Since $a \notin \uu(S')$, it follows from Claim~1 that $a \notin \uu(S)$. Now write $a = c + d$ for some $c,d \in S$. Because $a \in \mathcal{A}(S')$, either $c \in \uu(S')$ or $d \in \uu(S')$ and, therefore, Claim~1 guarantees that either $c \in \uu(S)$ or $d \in \uu(S)$. Thus, $a \in  \mathcal{A}(S)$. Hence $\mathcal{A}(S') \cap S \subseteq \mathcal{A}(S)$, as desired. Now pick $a' \in \mathcal{A}(S') \setminus S$ and write $a' = s + k(2b+u)$ for some $s \in S$ and $k \in \nn_0$. Since $a' \notin S$, we see that $k \ge 1$, and so the fact that $2b+u \notin \uu(S')$, in tandem with $a' \in \mathcal{A}(S')$, ensures that $k=1$ and $s \in \uu(S') = \uu(S)$. Thus, $a' \in 2b+u + \uu(S)$. Hence 
	\[
		\mathcal{A}(S') = \big( \mathcal{A}(S') \cap S \big) \sqcup \big( \mathcal{A}(S') \setminus S \big) \subseteq \mathcal{A}(S) \cup (2b + u + \mathcal{U}(S)),
	\]
	and so Claim~2 is established.
	\smallskip
	
	\noindent \textsc{Claim 3.} $b \notin \langle \mathcal{A}(S') \rangle$. 
	\smallskip
	
	\noindent \textsc{Proof of Claim 3.} Suppose, by way of contradiction, that $b \in \langle \mathcal{A}(S') \rangle$. Then in light of Claim~2 we can write
	\begin{equation} \label{eq:claim 3}
		b = (a_1 + \dots + a_j) + k(2b+u) + w
	\end{equation}
	for some $k \in \nn_0$, atoms $a_1, \dots, a_j \in \mathcal{A}(S)$, and a unit $w \in \uu(S)$. Since $b$ is not an atomic element of~$S$, we can assume that $k \ge 1$. Since $ku + (a_1 + \dots + a_k) + (2k-1)b + w = 0$, condition~(ii) above is not satisfied and, therefore, there exists $n_0 \in \nn$ such that $m_0 := n_0u \in S$. After multiplying both sides of the equality \eqref{eq:claim 3} by $n_0$, we find that $km_0 +n_0(a_1 + \dots + a_j) + n_0w + (2k-1)n_0 b = 0$, contradicting that $b \notin \uu(S)$. Therefore Claim~3 is also established.
	\smallskip

	With the previous three claims validated, we are now in a position to check that $2b+u \in S$. By definition,~$S'$ is a submonoid of $M$ containing $S$, and so Claim~3 guarantees that $S'$ belongs to~$\mathcal{S}$. By virtue of Claim~1, the equality $\uu(S) = \uu(S')$ holds, and this implies that $S \preceq S'$. As a result, the maximality of $S$ in the poset $\mathcal{S}$ ensures that $S' = S$. As a consequence, the containment $2b+u \in S$ holds, as desired.
	
	We are now ready to conclude our proof. Since $S$ is not atomic, the fact that $M$ is underatomic implies that $\gp(S)$ is a proper subgroup of $M$. Then we can take $v \in M \setminus \gp(S)$. As $v \notin \gp(S)$, it follows that $2b \pm v \notin S$. Since $2b+v \notin S$, neither condition~(i) nor condition~(ii) above hold for $v$ (playing the role of $u$): in particular the fact that condition~(i) does not hold implies that $nv \notin S$ for any $n \in \nn$. On the other hand, the fact that $2b-v \notin S$ ensures that none of the same two conditions above hold for $-v$ (playing the role of $u$): in particular, the fact that condition~(ii) does not hold allows us to take $n_0 \in \nn$ and $q_0 \in S$ such that $n_0(-v) + q_0 = 0$. However, the equality $n_0 v = q_0 \in S$ contradicts that $n v \notin S$ for any $n \in \nn$. This concludes our proof.
\end{proof}

For a monoid $M$, it is clear that if $\gp(M)$ is hereditarily atomic, then so is $M$. As the following example illustrates, the converse does not hold even in the class of rank-$1$ monoids.

\begin{example} \label{ex:HA monoid with non-HA group}
	Consider the monoid $M := \{0\} \cup \qq_{\ge 1}$ (under the standard addition). For each submonoid~$N$ of $M$, the fact that $0$ is not a limit point of $N^\bullet$ implies that $N$ is a BFM (by \cite[Proposition~4.5]{fG19}) and, therefore, $N$ is atomic. Hence $M$ is hereditarily atomic (this also follows from the fact that $M$ satisfies the ACCP via Corollary~\ref{cor:HA=ACCP for reduced monoids}). Now observe that $\gp(M) = \qq$, and $\qq$ has many non-atomic submonoids, the simplest one being $\qq_{\ge 0}$ (it follows from \cite[Theorem~3.2]{fG23} that the only torsion-free abelian group that is hereditarily atomic is $\zz$).
\end{example}

Hereditary atomicity in the context of integral domains was recently studied by Coykendall, Hasenauer, and the first author in~\cite{CGH21}. In the same paper, an integral domain $R$ is defined to be \emph{hereditarily atomic} provided that every subring of $R$ is atomic. Observe that if the multiplicative monoid $R^*$ is hereditarily atomic, then $R$ is a hereditarily atomic integral domain. The converse does not hold in general as the next two examples illustrate.

\begin{example} \label{ex:HAD not m-HA zero char}
	It is well known that every subring of the field $\qq$ can be obtained as a localization of $\zz$, and so every subring of $\qq$ is Noetherian. As Noetherian domains are atomic,~$\qq$ is a hereditarily atomic domain (for a characterization of fields that are hereditarily atomic, see~\cite[Theorem~4.4]{CGH21}). Let us exhibit a submonoid of the multiplicative group $\qq^*$ that is not atomic. Consider the submonoid $M := \qq_{\ge 1}$ of $\qq^*$. Since $1 = \min M$, we see that $M$ is reduced. Now for each $q \in M_{>1}$, we can take $n \in \nn$ large enough so that $\frac n{n+1} q > 1$, and so we can write~$q$ as a product of two non-units in $M$, namely, $q = \big(\frac{n+1}{n} \big) \cdot \big(\frac{n}{n+1}q \big)$. Therefore no element of $M$ is an atom, which implies that~$M$ is not atomic (as $M$ is not a group).
\end{example}

Let us take a look at another example, this time in positive characteristic.

\begin{example}  \label{ex:HAD not m-HA positive char}
	Consider the submonoid $M := \zz \times \nn_0$ of the abelian group $\zz^2$. Since $M/\uu(M)$ is isomorphic to $(\nn_0,+)$, the monoid $M$ satisfies the ACCP. Now consider the submonoid
	\begin{equation} \label{eq:nonnegative cone of Z^2}
		S := (\nn_0 \times \{0\}) \sqcup (\zz \times \nn)
	\end{equation}
	of $M$. Observe that $S$ is the nonnegative cone of $(\zz^2,+)$ under the lexicographical order with priority on the second coordinate. Therefore $S$ is reduced and the only atom of $S$ is the minimum of $S^\bullet$, namely, $(1,0)$. Therefore no element in $\zz \times \nn$ can be written as a sum of atoms in~$S$, and so $S$ is not atomic. Hence $M$ is not a hereditarily atomic monoid.
	
	Fix $p \in \pp$. It follows from \cite[Theorem~5.10]{CGH21} that the ring of Laurent polynomials $\ff_p[x^{\pm 1}]$ is a hereditarily atomic domain. We will argue that the multiplicative monoid $\ff_p[x^{\pm 1}]^*$ is not a hereditarily atomic monoid. We have seen in the previous paragraph that $S$ is not atomic. In addition, we can readily verify that the map given by the assignments $(m,n) \mapsto x^m (x+1)^n$ induces a monoid isomorphism from~$S$ to the submonoid $\{x^m (x+1)^n \mid (m,n) \in S\}$ of $\ff_p[x^{\pm 1}]^*$. Therefore the fact that $S$ is not atomic ensures that $\ff_p[x^{\pm 1}]^*$ is not a hereditarily atomic monoid.
\end{example}

\bigskip
\section{Ascending Chain Condition on Principal Ideals}
\label{sec:connection to the ACCP}

It was proved in~\cite{GV23} that in the class of torsion-free monoids, every hereditarily atomic monoid satisfies the ACCP. As the main result of this section, we generalize the mentioned result: we prove that every hereditarily atomic monoid (i.e., every underatomic monoid) satisfies the ACCP. First, we need to argue a technical lemma.
\smallskip

Fix $r \in \nn$. For $v,w \in \nn_0^r$, we write $v \preceq w$ if $w - v \in \nn_0^r$, that is, if $v$ divides $w$ in the additive monoid $\nn_0^r$. Clearly, $(\nn_0^r, \preceq)$ is a poset. Dickson's lemma, which is a useful result from combinatorics, states that every subposet of $(\nn_0^r, \preceq)$ has finitely many minimal elements. Before establishing the main result of this section, we need the following lemma, whose proof is based on Dickson's lemma.

\begin{lemma} \label{lem:HA implies ACCP aux}
	For any $r \in \nn$, if $X$ is an infinite subset of $\nn_0^r$, then there exists a sequence $(x_n)_{n \ge 1}$ with terms in $X$ such that $x_n \prec x_{n+1}$ for every $n \in \nn$.
\end{lemma}

\begin{proof}
	Fix an infinite subset $X$ of $\nn_0^r$. We construct the desired sequence inductively. For $v \in \nn_0^r$, we let $I_v$ denote the principal ideal of $(\nn_0^r,+)$ generated by $v$; that is,
	\[
		 I_v := v + \nn_0^r = \{x \in \nn_0^r : v \preceq x \}.
	\]
	By Dickson's lemma, $X$ has a nonempty finite set of minimal elements as a subposet of $(\nn_0^r, \preceq)$: let these minimal elements be $y_1, \dots, y_\ell$. Therefore $X = \bigcup_{i=1}^\ell I_{y_i} \cap X$. Since $|X| = \infty$, there exists $j \in \ldb 1, \ell \rdb$ such that $|I_{y_j} \cap X| = \infty$. After setting $x_1 := y_j$, we see that $|I_{x_1} \cap X| = \infty$. For the inductive step, suppose that we have found $x_1, \dots, x_n \in X$ with $x_1 \prec x_2 \prec \dots \prec x_n$ such that $|I_{x_n} \cap X| = \infty$. Set $X_n := (I_{x_n} \cap X) \setminus \{x_n\}$. Then $|X_n| = \infty$, and so it follows from Dickson's lemma that $X_n$ has a nonempty finite set of minimal elements as a subposet of $(\nn_0^r, \preceq)$: let these minimal elements be $z_1, \dots, z_m$. Since $X_n = \bigcup_{i=1}^m I_{z_i} \cap X_n$, the fact that $|X_n| = \infty$ ensures that $|I_{z_k} \cap X_n| = \infty$ for some $k \in \ldb 1, m \rdb$. Set $x_{n+1} := z_k$. Because $x_{n+1} \in X_n \subseteq I_{x_n} \setminus \{x_n\}$, it follows that $x_n \prec x_{n+1}$. Also, the inclusion $I_{x_{n+1}} \cap X_n \subseteq I_{x_{n+1}} \cap X$ implies that $|I_{x_{n+1}} \cap X| = \infty$, which completes our inductive step. As a result, we can guarantee the existence of a sequence $(x_n)_{n \ge 1}$ with terms in $X$ such that $x_n \prec x_{n+1}$ for every $n \in \nn$.
\end{proof}

We are in a position to prove that every hereditarily atomic monoid satisfies the ACCP.

\begin{theorem} \label{thm:HA implies ACCP}
	Every hereditarily atomic monoid satisfies the ACCP.
\end{theorem}

\begin{proof}
	Let $M$ be a hereditarily atomic monoid. 
	If $M$ is a group, then $M$ trivially satisfies the ACCP. Therefore we can assume, without loss of generality, that $M$ is not a group.
	
	Suppose, towards a contradiction, that $M$ does not satisfy the ACCP. Let $(m_n + M)_{n \ge 0}$ be an ascending chain of principal ideals that does not stabilize, and set $a_{n+1} := m_n - m_{n+1} \in M$. In this case, $m_n \notin \uu(M)$ for any $n \in \nn_0$. We can further assume that, for each $n \in \nn_0$, the inclusion $m_n + M \subseteq m_{n+1} + M$ is strict and, therefore, that none of the terms of the sequence $(a_n)_{n \ge 1}$ is invertible. We need the following claim.
	\smallskip
	
	\noindent \textit{Claim:} There exist two strictly increasing sequences $(s_j)_{j \ge 1}$ and $(t_j)_{j \ge 1}$ of positive integers satisfying the following two condition:
	\begin{enumerate}
		\item $s_j \le t_j < s_{j+1}$ for every $j \in \nn$, and
		\smallskip
		
		\item $m_0 \notin \bigoplus_{n \in \nn} \nn_0 (a_{s_n} + a_{s_n + 1} + \dots + a_{t_n})$.
	\end{enumerate}
	\smallskip
	
	\noindent \textit{Proof of Claim:} We proceed inductively. Suppose, by way of contradiction, that $m_0 \in \nn_0(a_1 + \dots + a_t)$ for every $t \in \nn$. Then there exists a sequence $(n_t)_{t \ge 1}$ of positive integers such that $m_0 = n_t(a_1 + \dots + a_t)$ for every $t \in \nn$. For each $t \in \nn$, the equality $(n_{t+1} - n_t)(a_1 + \dots + a_t) + n_{t+1} a_{t+1} = 0$, together with the fact that $a_1 \notin \uu(M)$, implies that $n_t \ge n_{t+1}$. Therefore there exists $k \in \nn$ such that $n_k = n_{k+1}$, which implies that $n_{k+1} a_{k+1} = 0$, contradicting that $a_{k+1} \notin \uu(M)$. Hence there exists $N \in \nn$ such that $m_0 \notin \nn_0 (a_1 + \dots + a_N)$. Thus, for our base case, we can take $s_1 = 1$ and $t_1 = N$.
	
	Now suppose we have already found $s_1, \dots, s_{\ell-1}$ and $t_1, \dots, t_{\ell - 1}$ in $\nn$ with $s_i \le t_i$ for every $i \in \ldb 1, \ell - 1 \rdb$ and $t_j < s_{j+1}$ for every $j \in \ldb 1, \ell - 2 \rdb$ such that $m_0 \notin \bigoplus_{i=1}^{\ell - 1} \nn_0 (a_{s_i} + a_{s_i + 1} + \dots + a_{t_i})$. For each $i \in \ldb 1, \ell-1 \rdb$, set
	\begin{equation} \label{eq:b_i's}
		b_i := a_{s_i} + a_{s_i + 1} + \dots + a_{t_i}.
	\end{equation}
	Now suppose, towards a contradiction, that $m_0 \in \bigoplus_{i=1}^\ell \nn_0 b_i$ for each $b_\ell = a_{s_\ell} + a_{s_\ell + 1} + \dots + a_{t_\ell}$, where $s_\ell, t_\ell \in \nn$ with $t_{\ell - 1} < s_\ell \le t_\ell$. Fix $s \in \nn$ with $t_{\ell - 1} < s$. Now, for each $k \in \nn_0$, set $c_k := a_s + a_{s+1} + \dots + a_{s+k}$ and then pick a vector $v_k := (n_1^{(k)}, \dots, n_\ell^{(k)}) \in \nn_0^\ell$ such that $m_0 = n_1^{(k)} b_1 + \dots + n_{\ell-1}^{(k)}b_{\ell-1} + n_\ell^{(k)} c_k$ (this is possible because $m_0 \in \big(\bigoplus_{i=1}^{\ell-1} \nn_0 b_i \big) \oplus \nn_0 c_k$). Observe that $n_\ell^{(k)} \neq 0$ for any $k \in \nn_0$ as, otherwise, $m_0 \in \bigoplus_{i=1}^{\ell-1} \nn_0 b_i$. Now consider the map $f_s \colon \nn_0 \to \nn_0^\ell$ defined by setting $f_s(k) := v_k$ for every $k \in \nn_0$. We claim that $f_s$ is injective. To argue this, suppose that $f_s(j) = f_s(k)$ for some $j,k \in \nn_0$ and assume, without loss of generality, that $j \le k$. Therefore $\big( \sum_{i=1}^{\ell - 1} n_i b_i \big) +n_\ell c_j = m_0 = \big( \sum_{i=1}^{\ell - 1} n_i b_i \big) + n_\ell c_k$ for some $(n_1, \dots, n_\ell) \in \nn_0^\ell$, and so $n_\ell(c_k - c_j) = 0$. From the fact that $n_\ell \neq 0$, we can now deduce that $c_j = c_k$, which in turn implies that $j = k$ because $a_{s+k} \notin \uu(M)$. Hence $f_s$ is injective. As a result, $|f_s(\nn_0)| = \infty$, and so Lemma~\ref{lem:HA implies ACCP aux} guarantees the existence of an infinite sequence $(x_n)_{n \ge 1}$ with terms in $f_s(\nn_0)$ such that $x_n \prec x_{n+1}$ for every $n \in \nn$. For each $n \in \nn$, set $k_n := f_s^{-1}(x_n)$. Observe that the sequence $(k_n)_{n \ge 1}$ cannot be decreasing as, otherwise, it would stabilize and, therefore, there would exist $i \in \nn$ such that $x_i = f_s(k_i) = f_s(k_{i+1}) = x_{i+1}$, which is not possible because the sequence $(x_n)_{n \ge 1}$ is strictly increasing in the poset $(\nn_0^\ell, \preceq)$. Thus, we can take $i \in \nn$ such that $k_i < k_{i+1}$. We can now rewrite the equality $\big( \sum_{j=1}^{\ell - 1} n_j^{(k_i)} b_j \big) +n_\ell^{(k_i)} c_{k_i} = m_0 = \big( \sum_{j=1}^{\ell - 1} n_j^{(k_{i+1})} b_j \big) +n_\ell^{(k_{i+1})} c_{k_{i+1}}$ as follows 
	\begin{equation} \label{eq:HA implies ACCP aux}
		\bigg( \sum_{j=1}^{\ell - 1} \Big( n_j^{(k_{i+1})} - n_j^{(k_i)} \Big) b_j \bigg) + \Big( n_\ell^{(k_{i+1})} - n_\ell^{(k_i)} \Big) c_{k_i}  + n_\ell^{(k_{i+1})} \bigg( \sum_{j=1}^{k_{i+1} - k_i} a_{s + k_i + j} \bigg) = 0.
	\end{equation}
	Since $\big( n_1^{(k_i)}, \dots, n_\ell^{(k_i)} \big) = x_i \prec x_{i+1} = \big( n_1^{(k_{i+1})}, \dots, n_\ell^{(k_{i+1})} \big)$, we can deduce from Equation~\eqref{eq:HA implies ACCP aux} that $n_\ell^{(k_{i+1})} \big( \sum_{j=1}^{k_{i+1} - k_i} a_{s + k_i + j} \big) \in \uu(M)$. This, together with the fact that $n_\ell^{(k_{i+1})} \neq 0$, implies that $a_{s+ k_i + 1} \in \uu(M)$, which is a contradiction. Hence there exist positive integers $s_\ell$ and $t_\ell$ with $t_{\ell - 1} < s_\ell \le t_\ell$ such that, after setting $b_\ell := a_{s_\ell} + a_{s_\ell + 1} + \dots + a_{t_\ell}$, we obtain that $m_0 \notin \bigoplus_{i=1}^\ell \nn_0 b_i$. This concludes our inductive construction, and so we can guarantee the existence of strictly increasing sequences $(s_j)_{j \ge 1}$ and $(t_j)_{j \ge 1}$ satisfying conditions~(1) and~(2) above. Thus, the claim follows.
	\smallskip
	
	Let $(s_j)_{j \ge 1}$ and $(t_j)_{j \ge 1}$ be two strictly increasing sequences satisfying conditions~(1) and~(2), and let $(b_n)_{n \ge 1}$ be the sequence whose terms are defined as in~\eqref{eq:b_i's}. 
	For each $n \in \nn$, the fact that $a_{s_n} \mid_M b_n$ ensures that $b_n \notin \uu(M)$. On the other hand, as $m_0 = m_n + \sum_{i=1}^n a_i$ for every $n \in \nn$, it follows that $\sum_{i=1}^n b_i \mid_M m_0$ for every $n \in \nn$. Now consider the sequence $(m'_n)_{n \ge 1}$ defined by setting $m'_n := m_0 - \sum_{i=1}^n b_i$ for every $n \in \nn$. Observe that each term of the sequence $(m'_n)_{n \ge 1}$ belongs to $M$. In addition, for each $n \in \nn$, the equality $m'_n = m'_{n+1} + b_{n+1}$ implies that $m'_n \notin \uu(M)$. Now consider the submonoid
	\[
		M' := \langle b_n, m'_n : n \in \nn \rangle
	\]
	of $M$. As $M$ is hereditarily atomic, $M'$ is atomic. For each $n \in \nn$, the fact that $b_n, m'_n \notin \uu(M)$ ensures that $b_n, m'_n \notin \uu(M')$. As a result, $M'$ is a reduced monoid. Also, for each $n \in \nn$, the equality $m'_n = m'_{n+1} + b_{n+1}$ implies that $m'_n \notin \mathcal{A}(M'_n)$. Hence $\mathcal{A}(M') \subseteq \{ b_n \mid n \in \nn \}$. Since $M'$ is atomic and $m_0 = b_1 + m'_1 \in M'$, we infer that $m_0 \in \bigoplus_{n \in \nn} \nn_0 b_n$, which is our desired contradiction. Hence $M$ must satisfy the ACCP.
\end{proof}

We could have coined the terms `under-ACCP' and `hereditary ACCP' with meanings mimicking those of underatomicity and hereditary atomicity for ACCP property. However, there is no need for this as the conditions referred by these terms would be equivalent notions.

\begin{cor}
	For a monoid $M$, the following conditions are equivalent.
	\begin{enumerate}
		\item[(a)] $M$ is hereditarily atomic.
		\smallskip
		
		\item[(b)] Every submonoid of $M$ satisfies the ACCP.
		\smallskip
		
		\item[(c)] Every undermonoid of $M$ satisfies the ACCP.
	\end{enumerate}
\end{cor}

\begin{proof}
	(a) $\Rightarrow$ (b): Assume that $M$ is hereditarily atomic, and let $N$ be a submonoid of $M$. Then $N$ is also hereditarily atomic and, therefore, it follows from Theorem~\ref{thm:HA implies ACCP} that $N$ satisfies the ACCP.
	\smallskip
	
	(b) $\Rightarrow$ (c): This is clear.
	\smallskip
	
	(c) $\Rightarrow$ (a): It follows from Theorem~\ref{thm:HA=UA} because every monoid satisfying the ACCP is atomic.
\end{proof}

The reverse implication of Theorem~\ref{thm:HA implies ACCP} does not hold in general. 

\begin{example} \label{ex:a non-atomic submonoid of Z^2}
	We have already seen in Example~\ref{ex:HAD not m-HA positive char} that the submonoid $\zz \times \nn_0$ of the abelian group $\zz^2$ satisfies the ACCP but its submonoid $ (\nn_0 \times \{0\}) \sqcup (\zz \times \nn)$ is not atomic. 
\end{example}

In the class consisting of reduced monoids the reverse implication of Theorem~\ref{thm:HA implies ACCP} holds. As a result, we obtain the following equivalences.

\begin{cor} \label{cor:HA=ACCP for reduced monoids}
	For a reduced monoid $M$, the following conditions are equivalent.
	\begin{enumerate}
	\item[(a)] $M$ is hereditarily atomic.
	\smallskip
	
	\item[(b)] $M$ satisfies the ACCP.
\end{enumerate}
\end{cor}

\begin{proof}
	(a) $\Rightarrow$ (b): This follows from Theorem~\ref{thm:HA implies ACCP}.
	\smallskip
	
	(b) $\Rightarrow$ (a): Since $M$ is reduced, from the fact that $M$ satisfies the ACCP, one can readily obtain that every submonoid of $M$ satisfies the ACCP and is, therefore, atomic.
\end{proof}

We conclude pointing out that there are integral domains having all its subrings satisfying the ACCP, while their multiplicative monoids are not even hereditarily atomic.

\begin{example}
	We have mentioned in Example~\ref{ex:HAD not m-HA zero char} that every subring of $\qq$ is Noetherian, which implies that every subring of $\qq$ satisfies the ACCP. On the other hand, we have argued in the same example that the multiplicative monoid $\qq^*$ contains a non-atomic submonoid.
\end{example}

\bigskip
\section{The Bounded Factorization Properties}
\label{sec:BF and FF}

According to Theorem~\ref{thm:HA implies ACCP}, every hereditarily atomic monoid satisfies the ACCP. Since the class consisting of monoids satisfying the ACCP contains every BFM, it is natural to wonder whether we can strength the statement of Theorem~\ref{thm:HA implies ACCP} by replacing the ACCP by the bounded factorization property. The following simple example provides a negative answer for this. 

\begin{example} \label{ex:underatomic not BFM}
	Consider the positive monoid $M = \big\langle \frac1p : p \in \pp \big\rangle$. It was outlined in \cite{AAZ90} that $M$ satisfies the ACCP (for a more detailed proof, see \cite[Example~3.3]{AG22} or \cite[Proposition~4.2(2)]{fG22}). This, along with the fact that $M$ is reduced, ensures that every submonoid of $M$ satisfies the ACCP and is, therefore, atomic (this also follows from Corollary~\ref{cor:HA=ACCP for reduced monoids}). Thus, $M$ is hereditarily atomic. However, one can readily argue that $\mathcal{A}(M) = \big\{ \frac1p : p \in \pp \big\}$, whence $\mathsf{L}(1) = \pp$. Thus, $M$ is not a BFM.
\end{example}

The following nonstandard definition will simplify the notation in the proof of the main result of this section, Theorem~\ref{thm:under-BFM = hereditarily BFM}.

\begin{definition}
	Let $M$ be a monoid. We say that an element $b \in M$ is \emph{boundedly atomic} if either $b \in \uu(M)$ or there exists $n \in \nn$ such that~$b$ can be written in $M$ as a sum of at most $n$ non-invertible elements. 
\end{definition}

Observe that a monoid $M$ is a BFM if and only if every element of $M$ is boundedly atomic. As for atomicity, we say that a monoid $M$ is a \emph{hereditary BFM} provided that every submonoid of $M$ is a BFM. Every hereditary BFM is a BFM by definition. Let $M$ be a monoid that is not a hereditary BFM, and take $b \in M$ that is not boundedly atomic in some submonoid of~$M$. The following set will be crucial in the proof of Theorem~\ref{thm:under-BFM = hereditarily BFM}:
\begin{equation} \label{eq:aux poset BF}
	\mathcal{S}_b := \big\{ S : S \text{ is a submonoid of } M \text{ and } b \in S \text{ is not boundedly atomic in } S \big\}.
\end{equation}
Since $b \in M$ is not boundedly atomic in some submonoid of~$M$, the set $\mathcal{S}_b$ is nonempty. Now consider the binary relation $\preceq_b$ on the set~$\mathcal{S}_b$ defined as follows: for any submonoids $S_1$ and $S_2$ contained in $\mathcal{S}_b$,
\begin{equation} \label{eq:aux binary relation BF}
	S_1 \preceq S_2 \quad \text{if} \quad S_1 \subseteq S_2 \ \ \text{and} \ \ \uu(S_1) = S_1 \cap \uu(S_2).
\end{equation}

As the following lemma indicates, it turns out that $\mathcal{S}_b$ is a nonempty poset under $\preceq_b$ and also that $\mathcal{S}_b$ contains a maximal element. The proof of the next lemma uses the same idea as that of Lemma~\ref{lem:poset with Zorn's lemma}, and so we will omit the parts that follow \emph{mutatis mutandis}.

\begin{lemma} \label{lem:poset with Zorn's lemma for BF}
	Let $M$ be a monoid, and let $b$ be an element of $M$ that is not boundedly atomic in some submonoid of~$M$. If $\mathcal{S}_b$ and $\preceq_b$ are defined as in \eqref{eq:aux poset BF} and~\eqref{eq:aux binary relation BF}, respectively, then the following statements hold.
	\begin{enumerate}
		\item $\mathcal{S}_b$ is a nonempty poset under the binary relation $\preceq_b$.
		\smallskip
		
		\item $\mathcal{S}_b$ has a maximal element. 
	\end{enumerate}
\end{lemma}

\begin{proof}
	(1) This can be done mimicking the proof of part~(1) of Lemma~\ref{lem:poset with Zorn's lemma}.
	\smallskip
	
	(2) To ease the notation set $\mathcal{S} := \mathcal{S}_b$ and $\preceq \, := \, \preceq_b$. In light of Zorn's lemma, it suffices to argue that every nonempty chain in the poset $\mathcal{S}$ has an upper bound. Let $(S_i)_{i \in I}$ be a nonempty chain in~$\mathcal{S}$, where $I$ is a nonempty set of indices. We claim that $S := \bigcup_{i \in I} S_i$ is an upper bound for $(S_i)_{i \in I}$ in~$\mathcal{S}$. Mimicking the first paragraph of the proof of part~(2) of Lemma~\ref{lem:poset with Zorn's lemma}, one can obtain that $S$ is a submonoid of $M$ such that
	\begin{equation} \label{eq:upper bound under-BFM}
		S_i \subseteq S \ \ \text{and} \ \ \uu(S_i) = S_i \cap \uu(S)
	\end{equation}
	for all $i \in I$. To argue that $S$ belongs to $\mathcal{S}$ we need to verify that $b$ is not boundedly atomic in $S$. To do this, fix $\ell \in \nn$. For any fixed $i \in I$, the fact that $b$ is not boundedly atomic in $S_i$ allows us to take $b_1, \dots, b_\ell \in S_i \setminus \uu(S_i)$ such that $b = b_1 + \dots + b_\ell$, and now the equality $\uu(S_i) = S_i \cap \uu(S)$ guarantees that $b_1, \dots, b_\ell \in S \setminus \uu(S)$. Hence $b$ is not boundedly atomic in $S$, which means that $S$ belongs to $\mathcal{S}$. Thus, $S$ is an upper bound for $(S_i)_{i \in I}$ in~$\mathcal{S}$, and so every nonempty chain in $\mathcal{S}$ has an upper bound, as desired.
\end{proof}

Instead of introducing the term `under-BFM', we will prove that a monoid $M$ is a hereditary BFM if and only if every undermonoid of $M$ is a BFM. Because the proof of the next theorem follows a similar strategy to that of Theorem~\ref{thm:HA=UA}, we will omit the parts of the proof we have already covered in the proof of Theorem~\ref{thm:HA=UA}.

\begin{theorem} \label{thm:under-BFM = hereditarily BFM}
	For a monoid $M$, the following conditions are equivalent.
	\begin{enumerate}
		\item[(a)] $M$ is a hereditary BFM.
		\smallskip
		
		\item[(b)] Every undermonoid of $M$ is a BFM.
	\end{enumerate}
\end{theorem}

\begin{proof}
	(a) $\Rightarrow$ (b): This implication follows directly from the definitions.
	\smallskip
	
	(b) $\Rightarrow$ (a): Assume that every undermonoid of $M$ is a BFM. Now suppose, by way of contradiction, that $M$ has a submonoid $N$ that is not a BFM. Then we can pick $b \in N \setminus \uu(N)$ such that $b$ is not boundedly atomic in $N$: for every $n \in \nn$, there exist $\ell \in \nn_{\ge n}$ and $b_1, \dots, b_\ell \in N \setminus \uu(N)$ such that
	\begin{equation} \label{eq:under-BFM}
		b = b_1 + \dots + b_\ell.
	\end{equation}
	Now we split the rest of the proof into the following two cases.
	\smallskip
	
	\noindent \textsc{Case 1:} $M$ is not a group. 
	Set $I := \{m \in M : m \nmid_M b \}$ and then set $M' := N \cup I$. We have already seen in the proof of Theorem~\ref{thm:HA=UA} that $I$ is a nonempty ideal of $M$ with $I \subseteq M \setminus \uu(M)$ and also that $M'$ is an undermonoid of~$M$ with $\uu(M') = \uu(N)$. Since for each $n \in \nn$ we can pick $\ell \in \nn_{\ge n}$ and $b_1, \dots b_\ell \in N \setminus \uu(N)$ such that the equality~\eqref{eq:under-BFM} holds, the inclusion $N \setminus \uu(N) \subseteq M' \setminus \uu(M')$ guarantees that we can also write $b$ as a sum of at least $n$ non-invertible elements in $M'$. Hence $M'$ is not a BFM, contradicting that every undermonoid of~$M$ is a BFM.
	\smallskip
	
	\noindent \textsc{Case 2:} $M$ is a group. Let $\mathcal{S}_b$ and $\preceq_b$ be as defined in \eqref{eq:aux poset BF} and \eqref{eq:aux binary relation BF}, respectively, and to ease the notation set $\mathcal{S} := \mathcal{S}_b$ and $\preceq \, := \, \preceq_b$. It follows from part~(1) of Lemma~\ref{lem:poset with Zorn's lemma for BF} that $\mathcal{S}$ is a nonempty poset under the binary relation $\preceq$, and it follows from part~(2) of the same lemma that $\mathcal{S}$ contains a maximal element, namely $S$. We proceed to prove the following claim.
	\smallskip
	
	\noindent \textsc{Claim.} Take $u \in M$ satisfying at least one of the following two conditions:
	\begin{enumerate}
		\item[(i)] there exists $n_0  \in \nn$ with $n_0 u \in S$;
		\smallskip
		
		\item[(ii)] $nu+q \neq 0$ for any $n \in \nn$ and $q \in S$.
	\end{enumerate}
	Then $2b+u \in S$.
	\smallskip
	
	\noindent \textsc{Proof of Claim.} Set $S' := S + \nn_0(2b + u)$. Mimicking the proof of Claim~1 in the proof of Theorem~\ref{thm:HA=UA}, we obtain that $2b + u \notin \uu(S')$ and also that $\uu(S') = \uu(S)$. Since $S$ is a submonoid of $S'$ with $\uu(S') = \uu(S)$, the fact that $b$ is not boundedly atomic in $S$ immediately implies that $b$ is not boundedly atomic in $S'$. As a result, $S'$ also belongs to $\mathcal{S}$. Now the maximality of $S$ in $\mathcal{S}$ enforces the equality $S' = S$ and, therefore, $2b+u \in S$. Hence the claim is established.
	\smallskip
	
	We are now in a position to conclude our proof. Since $S$ is not a BFM, the fact that every undermonoid of $M$ is a BFM ensures that $\gp(S)$ is a proper subgroup of~$M$. Now we can generate the final desired contradiction by mimicking the last paragraph of the proof of Theorem~\ref{thm:HA=UA} (which uses the previous claim).
\end{proof}

In the direction of Theorems~\ref{thm:HA=UA} and~\ref{thm:under-BFM = hereditarily BFM}, we have not been able to answer the question of whether a monoid whose undermonoids are all FFMs satisfies the property that whose submonoids are all FFMs. We pose the question here as an open problem for the interested readers.

\begin{question} \label{quest:under and hereditary FFM}
	Let $M$ be a monoid. If every undermonoid of $M$ is an FFM, does it follow that every submonoid of $M$ is an FFM?
\end{question}

The length-finite factorization property was recently introduced by A. Geroldinger and Q. Zhong in~\cite{GZ21}: this property, as the bounded factorization property, is a weaker natural version of the finite factorization property. Let $M$ be a monoid. For $\ell \in \nn$, we set
\[
	\mathsf{Z}_\ell(b) := \mathsf{Z}_{M,\ell}(b) := \{z \in \mathsf{Z}(b) : |z| = \ell \} \subseteq \mathsf{Z}(M).
\]
The monoid $M$ is called a \emph{length-finite factorization monoid} (LFFM) if $M$ is atomic and $|\mathsf{Z}_\ell(b)| < \infty$ for all $b \in M$ and $\ell \in \nn$. It follows directly from the definitions that $M$ is an FFM if and only if it is both a BFM and an LFFM. Thus, the length-finite factorization property somehow complements the bounded factorization property with respect to the finite factorization property. We have seen in Example~\ref{ex:underatomic not BFM} that not every hereditarily atomic monoid is a BFM. For the sake of completeness, let us take a look at a hereditarily atomic monoid that is not an LFFM.

\begin{example} \label{ex:underatomic not LFFM}
	Consider the monoid $M := \{0\} \cup \qq_{\ge 1}$ (under the standard addition). We have already seen in Example~\ref{ex:HA monoid with non-HA group} that $M$ is a hereditarily atomic monoid. Observe that $\mathcal{A}(M) = \qq \cap [1,2)$. Arguing that $M$ is not an LFFM amounts to observing that the element $3$ has infinitely many length-$2$ factorizations in $M$, which is illustrated by the equalities $3 = \big( \frac32 - \frac1n \big) + \big( \frac32 + \frac1n \big)$, for every $n \in \nn$ with $n \ge 3$. Hence $M$ is a hereditarily atomic monoid that is not an LFFM.
\end{example}

\begin{remark}
	As a monoid is an FFM if and only if it is both a BFM and an LFFM, it follows from Theorem~\ref{thm:under-BFM = hereditarily BFM} that a positive answer to Question~\ref{quest:under and hereditary FFM} follows from a positive answer to the question we obtain from it after replacing FFM by LFFM.
\end{remark}

The fact that every subring of a given integral domain $R$ is a BFD (resp., an LFFD, an FFD) does not imply that every submonoid of the multiplicative monoid of $R^*$ is a BFM (resp., an LFFM, an FFM). The following example sheds some light upon this observation.

\begin{example} \label{ex:HFFD not HFFM}
	We have mentioned in Example~\ref{ex:HAD not m-HA zero char} that every subring of $\qq$ can be obtained as a localization of $\zz$ and, therefore, every subring of $\qq$ is a Dedekind domain. In addition, it follows from combining \cite[Proposition~2.2]{AAZ90} and \cite[Proposition~1]{GW75} that every Dedekind domain is an FFD, and so every subring of $\qq$ must be an FFD. However, we know that the multiplicative monoid $\qq^*$ contains a non-atomic submonoid.
\end{example}

\bigskip
\section{Half-Factoriality and Length-factoriality}
\label{sec:HFM and LFM}

In this section we characterize the monoids all whose undermonoids are HFMs as well as the monoids all whose undermonoids are LFMs. We start by the half-factoriality characterization. Let $M$ be a monoid having at least one element $c$ with infinite order, and then take $b \in M$ that is not half-factorial in the submonoid $N := \nn_0 c$ of $M$ (for instance, we can take $b := 6c$ after observing that $\mathsf{L}_N(6c) = \{2,3\}$). Let $A_b$ denote the set of atoms dividing $b$ in $N$, and then consider the following set:
\begin{equation} \label{eq:aux poset HF}
	\mathcal{S}_b := \big\{ S : S \text{ is a submonoid of } M \text{ and } A_b \subseteq \mathcal{A}(S) \big\}.
\end{equation}
The set $\mathcal{S}_b$ is, therefore, nonempty. Now we consider the binary relation $\preceq_b$ on the set~$\mathcal{S}_b$ defined as follows: for any submonoids $S_1$ and $S_2$ contained in $\mathcal{S}_b$,
\begin{equation} \label{eq:aux binary relation HF}
	S_1 \preceq S_2 \quad \text{if} \quad S_1 \subseteq S_2 \ \ \text{and} \ \ \uu(S_1) = S_1 \cap \uu(S_2).
\end{equation}

\begin{lemma} \label{lem:poset with Zorn's lemma HF}
	Let $M$ be a monoid having an element with infinite order, and then take $b \in M$ that is not half-factorial in some submonoid $N$ of $M$. Let $A_b$ denote the set of atoms dividing $b$ in $N$, and let $\mathcal{S}_b$ and $\preceq_b$ be as in~\eqref{eq:aux poset HF} and~\eqref{eq:aux binary relation HF}. Then the following statements hold.
	\begin{enumerate}
		\item $\mathcal{S}_b$ is a nonempty poset under the binary relation $\preceq_b$.
		\smallskip
		
		\item $\mathcal{S}_b$ has a maximal element. 
	\end{enumerate}
\end{lemma}

\begin{proof}
	To ease the notation of this proof, let $\mathcal{S}$ and $\preceq$ stand for $\mathcal{S}_b$ and $\preceq_b$, respectively.
	\smallskip
	
	(1) It is clear that $\preceq$ is reflexive and antisymmetric, and transitivity follows as in the proof of part~(1) of Lemma~\ref{lem:poset with Zorn's lemma}.
	\smallskip
	
	(2) By Zorn's lemma, it suffices to argue that every nonempty chain in the poset $\mathcal{S}$ has an upper bound. Let $(S_i)_{i \in I}$ be a nonempty chain in~$\mathcal{S}$, where $I$ is a nonempty set of indices. We claim that $S := \bigcup_{i \in I} S_i$ is an upper bound for $(S_i)_{i \in I}$ in~$\mathcal{S}$. Mimicking the first paragraph of the proof of part~(2) of Lemma~\ref{lem:poset with Zorn's lemma}, one can obtain that $S$ is a submonoid of $M$ such that
	\begin{equation} \label{eq:upper bound under-HFM I}
		S_i \subseteq S \ \ \text{and} \ \ \uu(S_i) = S_i \cap \uu(S)
	\end{equation}
	for all $i \in I$. Also, observe that if we can write some $a \in A_b$ as $a = c+d$ for $c,d \in S$, then after taking $j \in I$ with $c,d \in S_j$ the fact that $a \in \mathcal{A}(S_j)$ ensures that either $c \in \uu(S_j) \subseteq \uu(S)$ or $d \in \uu(S_j) \subseteq \uu(S)$, whence $A_b \subseteq \mathcal{A}(S)$. As a result, $S$ belongs to $\mathcal{S}$, and so it follows from~\eqref{eq:upper bound under-HFM I} that~$S$ is an upper bound for $(S_i)_{i \in I}$ in~$\mathcal{S}$. Thus, every nonempty chain in $\mathcal{S}$ has an upper bound.
\end{proof}

We are in a position to prove that the monoids all its undermonoids are HFMs are precisely the abelian groups all whose elements have finite order.

\begin{theorem} \label{thm:undermonoids and the HFM}
	For a monoid $M$, the following conditions are equivalent.
	\begin{enumerate}
		\item[(a)] Every submonoid of $M$ is an HFM.
		\smallskip
		
		\item[(b)] Every undermonoid of $M$ is an HFM.
		\smallskip
		
		\item[(c)] $M$ is a group and each element of $M$ has finite order.
	\end{enumerate}
\end{theorem}

\begin{proof}
	(a) $\Rightarrow$ (b): This follows immediately from definitions.
	\smallskip
	
	(b) $\Rightarrow$ (c): Assume that every undermonoid of $M$ is an HFM. Let us prove first that $M$ has to be a group. Suppose, by way of contradiction that $M$ is not a group. Then take a non-invertible $c \in M$. Since~$c$ is not invertible, then the submonoid $\nn_0 c$ of $M$ is naturally isomorphic to $\nn_0$, and so $N := \langle 2c, 3c \rangle$ is an atomic submonoid with $\mathcal{A}(N) = \{2c, 3c\}$. Since $\mathsf{L}_N(6c) = \{2,3\}$, it follows that $N$ is not an HFM. Now set
	\[
		I := \{m \in M : m \nmid_M 6c \} \quad \text{ and } \quad M' := N \cup I.
	\]
	Using the fact that $M$ is not a group, we can argue, following the lines of the corresponding part of the proof of Theorem~\ref{thm:HA=UA}, that $M'$ is an undermonoid of $M$. From the fact that no element of $I$ can divide $6c$ in $M$, we obtain that the atoms dividing $6c$ in $N$, namely $2c$ and $3c$, are also atoms in $M'$. Therefore $\{2,3\} \subseteq  \mathsf{L}_{M'}(6c)$, which means that $M'$ is not an HFM, contradicting that every undermonoid of $M$ is an HFM. Therefore $M$ must be a group.
	
	We proceed to argue that each element of $M$ has finite order. Suppose, by way of contradiction, that there exists $c \in M$ with infinite order. Then the submonoid $N := \langle 2c, 3c \rangle$ of $M$ contains the element $b := 6c$ satisfying the inequality $|\mathsf{L}_N(b)| \ge 2$. Set
	\[
		A_b := \{a \in \mathcal{A}(N) : a \mid_N b\},
	\]
	and let $\mathcal{S}_b$ and $\preceq_b$ be as defined in \eqref{eq:aux poset HF} and \eqref{eq:aux binary relation HF}, respectively. To simplify our notation, let  $\mathcal{S}$ and $\preceq$ stand for $\mathcal{S}_b$ and $\preceq_b$, respectively. It follows from part~(1) of Lemma~\ref{lem:poset with Zorn's lemma HF} that $\mathcal{S}$ is a nonempty poset under the binary relation $\preceq$, and it follows from part~(2) of the same lemma that $\mathcal{S}$ contains a maximal element, namely $S$. We proceed to prove the following claim.
	\smallskip
	
	\noindent \textsc{Claim.} Take $u \in M$ satisfying at least one of the following two conditions:
	\begin{enumerate}
		\item[(i)] there exists $n_0  \in \nn$ with $n_0 u \in S$;
		\smallskip
		
		\item[(ii)] $nu+q \neq 0$ for any $n \in \nn$ and $q \in S$.
	\end{enumerate}
	Then $2b+u \in S$.
	\smallskip
	
	\noindent \textsc{Proof of Claim.} Set $S' := S + \nn_0(2b + u)$. Since $A_b$ is a nonempty subset of $S$, it follows that $b \notin \uu(S)$ and so, after mimicking the proof of Claim~1 in the proof of Theorem~\ref{thm:HA=UA}, we obtain that $2b + u \notin \uu(S')$ and also that $\uu(S') = \uu(S)$. Let us argue now that $A_b \subseteq \mathcal{A}(S')$. To do this, fix $a \in A_b$. Since $A_b \subseteq \mathcal{A}(S)$, we can write $b = a+s$ for some $s \in S$. Suppose that $a = \big(s_1 + k_1(2b+u)\big) + \big(s_2 + k_2(2b+u)\big)$ for some $s_1, s_2 \in S$ and $k_1, k_2 \in \nn_0$. If $k_1 + k_2$ were positive, then condition~(ii) above would not hold as the following equality
	\begin{equation} \label{eq:aux equation HFM}
		s + (2(k_1 + k_2) - 1)b + s_1 + s_2 + (k_1 + k_2) u = 0
	\end{equation}
	illustrates, and so there must exist $n_0 \in \nn$ such that $m_0 := n_0 u \in S$: then after multiplying both sides of~the equality \eqref{eq:aux equation HFM} by $n_0$, we obtain that $n_0s + (2(k_1 + k_2) - 1)n_0b + n_0(s_1 + s_2) + (k_1 + k_2) m_0 = 0$, which is not possible because $b \notin \uu(S)$. Hence $k_1 = k_2 = 0$ and so the equality $a = s_1 + s_2$, in tandem with $a \in \mathcal{A}(S)$, implies that either $s_1 \in \uu(S) = \uu(S')$ or $s_2 \in \uu(S) = \uu(S')$. As a consequence, $a \in \mathcal{A}(S')$, and we conclude that $A_b \subseteq \mathcal{A}(S')$, as desired. This implies that $S'$ also belongs to $\mathcal{S}$. Therefore the equality $\uu(S') = \uu(S)$ means that $S \preceq S'$, and so the maximality of $S$ in $\mathcal{S}$ guarantees that $S' = S$. Thus, $2b+u \in S$, and the claim follows.
	\smallskip
	
	Since $A_b \subseteq \mathcal{A}(S)$, it follows that $|\mathsf{L}_S(b)| \ge 2$ and so that $S$ is not an HFM. This, together with the fact that every undermonoid of $M$ is an HFM, implies that $\gp(S)$ is a proper subgroup of~$M$. Now we can generate the final desired contradiction following the last paragraph of the proof of Theorem~\ref{thm:HA=UA} \emph{mutatis mutandis} (this needs the established claim).
	\smallskip
	
	(c) $\Rightarrow$ (a): Suppose that $M$ is a group with no elements of infinite order. Then every submonoid of~$M$ must be a group and, therefore, an HFM in a trivial way.
\end{proof}

We conclude this section characterizing monoids all whose undermonoids are length-factorial monoids. A monoid $M$ is called a \emph{length-factorial monoid} (LFM) if $M$ is atomic and $|\mathsf{Z}_\ell(b)| \le 1$ for all $b \in M$ and $\ell \in \nn$. It follows directly from the definitions that a monoid is a UFM if and only if it is both an HFM and an LFM. The notion of length-factoriality was introduced by Coykendall and Smith in \cite{CS11} as a dual of half-factoriality (with respect to the unique factorization property) and under the term `other-half-factoriality': in the same paper, they proved that the multiplicative monoid of an integral domain $R$ is an LFM if and only if $R$ is a UFD. Length-factoriality was further investigated in~\cite{CCGS21}, where the term `length-factoriality' was adopted. Further studies of length-factoriality have been carried out in~\cite{GZ21} and more recently in~\cite{BVZ23}. 

It turns out that, as for the case of half-factoriality, all the undermonoids of a monoid are LFMs if and only if the given monoid is an abelian group containing no elements of infinite order.

\begin{theorem} \label{thm:undermonoids and the LFM}
	For a monoid $M$, the following conditions are equivalent.
	\begin{enumerate}
		\item[(a)] Every submonoid of $M$ is an LFM.
		\smallskip
		
		\item[(b)] Every undermonoid of $M$ is an LFM.
		\smallskip
		
		\item[(c)] $M$ is a group and each element of $M$ has finite order.
	\end{enumerate}
\end{theorem}

\begin{proof}
	(a) $\Rightarrow$ (b): This follows directly from definitions.
	\smallskip
	
	(b) $\Rightarrow$ (c): For this, assume that every undermonoid of $M$ is an LFM, and let us argue first that $M$ is a group. Suppose, towards a contradiction, that this is not the case. Now take a non-invertible $c \in M$. Since $c$ is not invertible, $N := \langle 3c, 4c, 5c \rangle$ is an atomic submonoid of $M$ with $\mathcal{A}(N) = \{3c, 4c, 5c\}$. Because~$N$ is a reduced monoid, the element $8c$ has exactly two factorizations in $N$, namely, $3c + 5c$ and $4c + 4c$. Thus, $|\mathsf{Z}_{N,2}(8c)| = 2$, and so $N$ is not an LFM. We can now follow the lines of the corresponding part of the proof of Theorem~\ref{thm:undermonoids and the HFM} (with $8c$ playing the role of $6c$) to construct an undermonoid $M'$ of~$M$ containing $N$ such that $\uu(M') = \uu(N) = \{0\}$ and $\mathcal{A}(N) \subseteq \mathcal{A}(M')$. In this case, observe that $|\mathsf{Z}_{N,2}(8c)| \ge 2$. Therefore $M'$ is not an LFM, which contradicts that every undermonoid of $M$ is an LFM. Hence $M$ must be a group.
	
	We proceed to argue that each element of $M$ has finite order. Suppose, by way of contradiction, that there exists an element $c \in M$ with infinite order, and define $N$ as in the previous paragraph, namely, $N := \langle 3c, 4c, 5c \rangle$. After setting $b := 8c$, we see that
	\[
		A_b := \{a \in \mathcal{A}(N) : a \mid_N b\} = \mathcal{A}(N) = \{3c, 4c, 5c\}.
	\]
	Now we can argue the existence of a maximal element $S$ in the poset $(\mathcal{S}_b, \preceq_b)$, defined as in~\eqref{eq:aux poset HF} and~\eqref{eq:aux binary relation HF}, by following the lines of the corresponding part of the proof of Theorem~\ref{thm:undermonoids and the HFM}, and we can also prove the following claim as we did in the proof of Theorem~\ref{thm:undermonoids and the HFM}.
	
	\noindent \textsc{Claim.} Take $u \in M$ satisfying at least one of the following two conditions:
	\begin{enumerate}
		\item[(i)] there exists $n_0  \in \nn$ with $n_0 u \in S$;
		\smallskip
		
		\item[(ii)] $nu+q \neq 0$ for any $n \in \nn$ and $q \in S$.
	\end{enumerate}
	Then $2b+u \in S$.
	\smallskip
	
	As $S$ belongs to $\mathcal{S}_b$, the inclusion $A_b \subseteq \mathcal{A}(S)$ holds. In addition, observe that no two of the atoms $3c, 4c$, and $5c$ (in $A_b$) can be associates in $S$ as otherwise either $\{\pm s\} \subset S$ or $\{\pm 2s\} \subset S$ (the first inclusion is not possible because $3s \in \mathcal{A}(S)$ and the second inclusion is not possible because $4s \in \mathcal{A}(S)$). Since $3s, 4s$, and $5s$ are pairwise non-associate atoms of $S$, it follows that $|\mathsf{Z}_{S,2}(b)| \ge 2$, which implies that $S$ is not an LFM. This, together with the fact that every undermonoid of $M$ is an LFM, implies that $\gp(S)$ is a proper subgroup of~$M$. In order to generate our desired contradiction, we just need to mimic the last paragraph of the proof of Theorem~\ref{thm:HA=UA} (which uses the previous claim).
	\smallskip
	
	(c) $\Rightarrow$ (a): Suppose that $M$ is a group containing no elements of infinite order. Then every submonoid of $M$ must be a group and, therefore, an LFM in a trivial way.
\end{proof}

For the sake of completeness, we conclude this paper emphasizing the existence of an integral domain~$R$ whose subrings are all HFDs (resp., LFDs) such that the multiplicative monoid $R^*$ contains a submonoid that is not an HFM (resp., an LFM). Let us begin with the half-factorial property.

\begin{example}
	Since the integral domain $\zz$ does not have proper subring, it follows trivially that every subring of $\zz$ is a UFD and, therefore, an HFD. We proceed to identify a submonoid of $\zz^*$ that is not an HFM. To do this, consider the submonoid $M := \{(0,0)\} \cup \nn^2$ of the abelian group $\zz^2$. One can readily check that $M$ is atomic with
	\[
		\mathcal{A}(M) := \{(m,n) \in \nn^2 : m=1 \text{ or } n=1 \}. 
	\]
	Now fix $p,q \in \pp$ with $p \neq q$, and let $S$ be the image of the monoid homomorphism $\varphi \colon M \to \nn$ defined by the assignments $(m,n) \mapsto p^m q^n$. Since $\nn$ is a UFM, it follows that $S$ is an isomorphic copy of $M$. Because $\varphi \colon M \to S$ is an isomorphism, 
	\[
		\mathcal{A}(S) := \big\{2^m 3^n : m=1 \text{ or } n=1 \big\}.
	\]
	As a consequence, $(pq^2) \cdot (p^2 q)$ and $(pq) \cdot (pq) \cdot (pq)$ are two factorizations of $p^3q^3$ in $S$ having different lengths. Thus, $S$ is a submonoid of the multiplicative monoid $\zz^*$ that is not an HFM.
\end{example}

Let us now consider the length-factorial property.

\begin{example}
	Let us now identify a submonoid of the multiplicative monoid $\zz^*$ that is not an LFM. Consider the Hilbert monoid $H$, that is, $H := (4 \nn_0 + 1, \cdot)$. It is well known (see \cite{sC14}) and routine to argue that the set of atoms of $H$ is $\mathcal{A}(H) := P \sqcup Q$, where
	\[
		P := \big\{ p \in \pp : \ p \equiv 1 \! \! \! \pmod{4}\big\} \quad \text{ and } \quad Q := \big\{ p_1 p_2 : \ p_1, p_2 \in \pp \ \text{ and } \ p_1, p_2 \equiv 3 \! \! \! \pmod 4 \big\}.
	\]
	As a result, $H$ is a submonoid of the multiplicative monoid $\zz^*$ that is not an LFM: indeed, because $H$ is reduced, for any two distinct primes $p$ and $q$ with $p,q \equiv 3 \! \pmod 4$ we see that $p^2 \cdot q^2$ and $(pq) \cdot (pq)$ are two distinct length-$2$ factorizations of $p^2q^2$ in $H$.
\end{example}

\bigskip
\section*{Acknowledgments}

The first author was kindly supported by the NSF, under the awards DMS-1903069 and DMS-2213323.

\bigskip
\section*{Conflict of Interest Statement}

The authors state that there is no conflict of interest.

\bigskip
\section*{Data availability statement}

The authors state that this manuscript has no associated data.

\bigskip


\begin{thebibliography}{20}
	
	\bibitem{AAZ90} D.~D. Anderson, D.~F. Anderson, and M. Zafrullah, \emph{Factorizations in integral domains}, J. Pure Appl. Algebra \textbf{69} (1990) 1--19.
	
	\bibitem{AAZ92} D. D. Anderson, D. F. Anderson, and M. Zafrullah, \emph{Factorization in integral domains II}, J. Algebra \textbf{152} (1992) 78–93.

	\bibitem{AG22} D. F. Anderson and F. Gotti, \emph{Bounded and finite factorization domains}. In: Rings, Monoids, and Module Theory (Eds. A. Badawi and J. Coykendall) pp. 7--57. Springer Proceedings in Mathematics \& Statistics, Vol. \textbf{382}, Singapore, 2022.

	\bibitem{BC19} J. G. Boynton and J. Coykendall, \emph{An example of an atomic pullback without the ACCP}, J. Pure Appl. Algebra \textbf{223} (2019) 619--625.
	
	\bibitem{BVZ23} A. Bu, J. Vulakh, and A. Zhao, \emph{Length-factoriality and pure irreducibility}, Comm. Algebra \textbf{51} (2023) 3745--3755.
	
	\bibitem{lC60} L. Carlitz, \emph{A characterization of algebraic number fields with class number two}, Proc. Amer. Math. Soc. \textbf{11} (1960) 391--392.
	
	\bibitem{sC14} S. T. Chapman, \emph{A tale of two monoids: a friendly introduction to nonunique factorizations}, Math. Mag. \textbf{87} (2014) 163--173.
	
	\bibitem{CCGS21} S. T. Chapman, J. Coykendall, F. Gotti, and W. W. Smith, \emph{Length-factoriality in commutative monoids and integral domains}, J. Algebra \textbf{578} (2021) 186--212.

	\bibitem{pC68} P.~M. Cohn, \emph{Bezout rings and and their subrings}, Math. Proc. Cambridge Philos. Soc. \textbf{64} (1968) 251--264.

	\bibitem{CG24} J. Coykendall and F. Gotti, \emph{Atomicity in integral domains}. In: Rings and Factorizations (Eds. M. Brešar, A. Geroldinger, B. Olberding, and D. Smertnig), Springer Nature, Switzerland, 2024.

	\bibitem{CG19} J. Coykendall and F. Gotti, \emph{On the atomicity of monoid algebras}, J. Algebra \textbf{539} (2019) 138--151.
	
	\bibitem{CGH21} J. Coykendall, F. Gotti, and R. Hasenauer, \emph{Hereditary atomicity in integral domains}. J. Pure Appl. Algebra \textbf{227} (2023) 107249.
	
	\bibitem{CS11} J. Coykendall and W. W. Smith, \emph{On unique factorization domains}, J. Algebra \textbf{332} (2011) 62--70.
	
	\bibitem{GH06} A.~Geroldinger and F.~Halter-Koch, \emph{Non-unique Factorizations: Algebraic, Combinatorial and Analytic Theory}, Pure and Applied Mathematics Vol. 278, Chapman \& Hall/CRC, Boca Raton, 2006.

	\bibitem{GZ21} A. Geroldinger and Q. Zhong, \emph{A characterization of length-factorial Krull monoids}, New York J. Math. \textbf{27} (2021) 1347--1374.
	

	
	\bibitem{fG23} F.~Gotti, \emph{Hereditary atomicity and ACCP in abelian groups}. Submitted. Preprint on arXiv: https://arxiv.org/abs/2303.01039.
		
	\bibitem{fG19} F.~Gotti, \emph{Increasing positive monoids of ordered fields are FF-monoids}, J. Algebra \textbf{518} (2019) 40--56.

	\bibitem{fG22} F. Gotti, \emph{On semigroup algebras with rational exponents}, Comm. Algebra \textbf{50} (2022) 3--18.
	
	\bibitem{GL23} F. Gotti and B. Li, \emph{Atomic semigroup rings and the ascending chain condition on principal ideals}, Proc. Amer. Math. Soc. \textbf{151} (2023) 2291--2302.

	\bibitem{GL23a} F. Gotti and B. Li, \emph{Divisibility and a weak ascending chain condition on principal ideals}. Preprint on arXiv: https://arxiv.org/abs/2212.06213

	\bibitem{GR25} F. Gotti and H. Rabinovitz, \emph{On the ascent of atomicity to monoid algebras}, J. Algebra \textbf{663} (2025) 857--881.

	\bibitem{GV23} F. Gotti and J. Vulakh, \emph{On the atomic structure of torsion-free monoids}, Semigroup Forum \textbf{107} (2023) 402--423.
	
	\bibitem{aG74} A. Grams, \emph{Atomic rings and the ascending chain condition for principal ideals}, Math. Proc. Cambridge Philos. Soc. \textbf{75} (1974) 321--329.

	\bibitem{GW75} A. Grams and H. Warner, \emph{Irreducible divisors in domains of finite character}, Duke Math. J. \textbf{42} (1975) 271--284.

	\bibitem{pG01} P.~A. Grillet, \emph{Commutative Semigroups}, Advances in Mathematics Vol. 2, Kluwer Academic Publishers, Boston, 2001.

	\bibitem{fHK92} F. Halter-Koch, \emph{Finiteness theorems for factorizations}, Semigroup Forum \textbf{44} (1992) 112--117.
	
		


	\bibitem{aZ76} A. Zaks, \emph{Half-factorial domains}, Bull. Amer. Math. Soc. \textbf{82} (1976) 721--723.
	
	\bibitem{aZ80} A.~Zaks, \emph{Half-factorial domains}, Israel J. Math. \textbf{37} (1980) 281--302.

\end{thebibliography}
\end{document}